\documentclass{amsart}
\usepackage{amsfonts,amssymb,amscd,amsmath,enumerate,verbatim,calc}
\usepackage[all]{xy}

\newcommand{\CM}{Cohen-Macaulay}

\newcommand{\Ic}{\mathcal{I} }

\newcommand{\m}{\mathfrak{m} }

\newcommand{\Rc}{\mathcal{R} }

\newcommand{\C}{\mathcal{C} }

\newcommand{\AR}{\mathcal{AR} }

\newcommand{\Sc}{\mathcal{S} }

\newcommand{\Z}{\mathbb{Z} }
\newcommand{\Q}{\mathbb{Q} }

\newcommand{\rt}{\rightarrow}

\newcommand{\ov}{\overline}

\newcommand{\rank}{\operatorname{rank}}

\newcommand{\add}{\operatorname{add}}

\newcommand{\charp}{\operatorname{char}}

\newcommand{\CMa}{\operatorname{CM}}

\newcommand{\Hom}{\operatorname{Hom}}

\newcommand{\Ext}{\operatorname{Ext}}
\newcommand{\md}{\operatorname{mod}}

\theoremstyle{plain}

\newtheorem{theorem}{Theorem}[section]
\newtheorem{corollary}[theorem]{Corollary}
\newtheorem{lemma}[theorem]{Lemma}
\newtheorem{proposition}[theorem]{Proposition}

\theoremstyle{definition}
\newtheorem{definition}[theorem]{Definition}
\newtheorem{s}[theorem]{}
\newtheorem{remark}[theorem]{Remark}

\theoremstyle{remark}

\numberwithin{equation}{section}
\begin{document}

\title{On $G(A)_\Q$ of rings of finite representation type}
\author{Tony~J.~Puthenpurakal}
\date{\today}
\address{Department of Mathematics, IIT Bombay, Powai, Mumbai 400 076}

\email{tputhen@math.iitb.ac.in}

 \begin{abstract}
 Let $(A,\m)$ be an excellent Henselian \CM \ local ring of finite representation type. 
 If the AR-quiver of $A$  is known then  by a result of Auslander and Reiten   one can explicity  compute $G(A)$ the Grothendieck group of finitely generated $A$-modules. If the AR-quiver is not known then in this paper we 
 give estimates of $G(A)_\Q = G(A)\otimes_\Z \Q$ when $k = A/\m$ is perfect. As an application we prove that if $A$ is an excellent equi-characteristic  Henselian Gornstein local ring of positive even dimension with $\charp A/\m \neq 2,3,5$ (and $A/\m$ perfect) then  $G(A)_\Q  \cong \Q$.
\end{abstract}
 \maketitle

\section{introduction}
Let $(A,\m)$ be a Henselian Noetherian local ring. Then it is well-known that the category of finitely generated $A$-modules satisfy the Krull-Schmidt property,  i.e., every finitely generated $A$-module is uniquely
a direct sum of indecomposable $A$-modules (with local endomorphism rings).  Now assume that $A$ is \CM.  Then we say $A$ is of finite
(\CM) representation type if $A$ has only finitely many indecomposable maximal \CM \ $A$-modules upto isomorphism. 
To study
(not necessarily commutative) Artin algebra's
 Auslander and Reiten introduced the theory  of almost-split sequences.
These are now called AR-sequences. 
Later Auslander and Reiten extended the theory of AR-sequences to the case of commutative Henselian isolated singularities.
Good references for this topic are \cite{Y} and \cite{LW}. Let $\CMa(A)$ denote the full subcategory of maximal \CM \  (= MCM) $A$-modules.

\begin{remark}
Note we can define Grothendieck group of any extension closed subcategory $\mathcal{S}$ of $\md (A)$ the category of all finitely generated $A$-modules, we denote it by $G(\mathcal{S})$. By \cite[13.2]{Y} the natural map $G(\CMa(A)) \rt G(\md(A)$ is an isomorphism. Throughout this section we work with $G(\CMa(A))$ and by abuse of notation denote it by $G(A)$. 
\end{remark}

Let $A$ be a Henselian  \CM \ local of finite representation type. Set $\Ic_A$ to be the set of all indecomposable MCM  $A$-modules upto isomorphism.
If $W$ is a subset of $\Ic_A$ then set $\add(W)$ be the set consisting of finite direct sums of elements of $W$.
 Also let $\AR(A)$ denote the set of all AR-sequences in $A$ upto isomorphism.
Let $F(A)$ be the free abelian group generated on $\add(\Ic_A)$. Let $\AR_0(A)$ be the subgroup of $F(A)$ generated by 
$$\{ X_1-X_2 + X_3 \mid  \text{ there is a sequence} \  0 \rt X_1 \rt X_2 \rt X_3 \rt 0 \ \text{in} \ \AR(A) \}.$$
By a result due to Auslander-Reiten \cite[2.2]{AR} (also see \cite[13.7]{Y}) we have  $G(A) = F(A)/\AR_0(A)$. 

Computing AR-sequences is usually a tedious task and usually we assume $A$ is equi-characteristic, complete with algebraically closed residue field. The main objective of this paper is to give estimates of rank of $G(A)$ when the residue field is perfect but not necessarily algebraically closed. 

In our  introduction let us assume $(A,\m)$ is an excellent Henselian \\  \CM \ local of finite representation type and containing a field isomorphic to $k = A/\m$ which we also denote by $k$. (In our proofs we will deal with a more general case). 
We assume $k$ is perfect.
Let $\ov{k}$ be the algebraic closure of $k$. Let  
$$\C_k = \{E \mid E \ \text{is a finite extension of } \ k, \ \text{and} \ E \subseteq \ov{k} \}.$$
For each $E$ in $\C_k$ set $A^E = A\otimes_k E$. We have an obvious directed system of rings $\{ A^E \}_{E \in \C_k}$. Set $T = \lim_{E \in \C_k} A^E = A\otimes_k \ov{k}$. Then $A^E, T$ are excellent Henselian \CM \ local of finite representation type. So $\widehat{T}$ the completion of $T$ is \CM \ of finite representation type. 
It is not difficult to show $G(T) \cong G(\widehat{T})$. 
For $k \subseteq F \subseteq E $, where $E, F \in \C_k$, 
we have an obvious map
$\eta^E_F \colon G(A^F) \rt G(A^E)$  given by $M \rt M\otimes_{A^F} A^E $.
It is clear that we have a direct system of abelian groups $\{G(A^E)\}_{E \in \C_k}$. So we have 
an abelian group $\lim_{E \in \C_k}G(A^E) $ and natural maps $\eta_E \colon G(A^E) \rt \lim_{E \in \C_k}G(A^E)$.
 Let $E \in \C_k$. As $T$ is a flat $A^E$-algebra we have an obvious map
$\xi_E \colon G(A^E) \rt G(T)$  given by $M \rt M\otimes_{A^E} T $.  The maps $\xi_E$ are compatible with $\eta^E_F$ whenever $k \subseteq F \subseteq E$.
So we have a natural map
\[
\xi \colon \lim_{E \in \C_k}G(A^E) \rt G(T).
\]
Our main result is
\begin{theorem}
\label{main-intro} $\xi$ is an isomorphism.
\end{theorem}

Theorem \ref{main-intro} does not give us any estimates on $G(A)$. If $H$ is an abelian group then we set $H_\Q = H \otimes_\Z \Q$ and if $f \colon H \rt L$ is a homomorphism of abelian groups then we set $f_\Q$ to be map from $H_\Q \rt L_\Q$ induced by $f$. 

It is well-known that direct limits commutes with tensor products, see \cite[Theorem A1, p.\ 270]{Mat}. So we have an isomorphism
\[
\xi_\Q \colon \lim_{E \in \C_k}G(A^E)_\Q \rt G(T)_\Q.
\]
Our next result essentially an observation.
\begin{proposition}\label{inj-intro}
Let $F \in \C_k$ be any. Then the map 
$$(\eta_F)_\Q \colon G(A^F)_\Q \rt \lim_{E \in \C_k}G(A^E)_\Q$$
 is an injection. In particular we have an injection from $G(A)_\Q$ to $G(T)_\Q$.
\end{proposition}

Complete equi-characteristic Gorenstein local rings of finite representation type (with   $\charp A/\m \neq 2,3,5$) are precisely the ADE-singularities, see  \cite[9.8]{LW}. Furthermore in this case their AR-quiver is known and so their Grothendieck groups have been computed, see \cite[13.10]{Y}. As an easy consequence to our results we show that 
\begin{corollary}\label{intro-cor}
Let $(A,\m)$ be an excellent equi-characteristic Henselian Gorenstein local ring of finite representation type.
Assume $k$ is perfect.
Then
\begin{enumerate}[\rm (1)]
\item
If 
$\dim A$ is positive and even then $G(A)_\Q \cong \Q$.
\item 
If 
$\dim A$ is odd  then $ \dim_\Q G(A)_\Q \leq 3$.
\end{enumerate}
\end{corollary}
In fact in (1) we have $G(A)_\Q \cong G(\widehat{T})_\Q$. In section five we give an example which shows that in (2);
$G(A)_\Q$ can be a proper subspace of $ G(\widehat{T})_\Q$. The same example shows that we cannot in general 
compare torsion of $G(A)$ with torsion of $ G(\widehat{T})$.

We now describe in brief the contents of this paper. In section two we discuss some preliminaries on AR-sequences that we need. In section three we describe our construction. In section four we prove Theorem \ref{main-intro}, Proposition \ref{inj-intro} and Corollary \ref{intro-cor}. Finally in section five we discuss an example which shows that torsion does not behave well with our construction. 

\emph{Convention}: Throughout this paper all  rings are commutative Noetherian and all modules (unless stated otherwise) are finitely generated.
\section{Some preliminaries on Auslander-Reiten sequences}
In this section we discuss some preliminaries on Auslander-Reiten (AR) sequences that we need. The reference for this section is \cite[Chapter 2]{Y}.  Let $(A, \m)$ be a Henselian \CM \ local ring.

\s 
Let $M \in \CMa(A)$ be indecomposable.
We define a set of short exact sequences in $\CMa(A)$ as follows:
\[
\Sc(M) = \{ s \colon 0 \rt N_s \rt E_s \rt M \rt 0 \mid N_s \  \text{is indecomposable and} \ s \ \text{is non-split}\}.
\] 
If $M$ is non-free then $\Sc(M)$ is non-empty, \cite[2.2]{Y}. Define a partial order $>$ on $\Sc(M)$ as follows. Let $s,t \in \Sc(M)$. Then we say $s > t$ if there is
$f \in \Hom_A(N_s, N_t)$ such that $\Ext^1_A(M,f)(s) = t$. 
This is equivalent to the existence of 
 a commutative diagram:
\[
  \xymatrix
{
 0
\ar@{->}[r]
 & N_s
\ar@{->}[r]
\ar@{->}[d]^{f}
& E_s
\ar@{->}[r]
\ar@{->}[d]
&M
\ar@{->}[d]^{j}
\ar@{->}[r]
&0
\\
 0
\ar@{->}[r]
 & N_t
\ar@{->}[r]
& E_t
    \ar@{->}[r]
    &M
      \ar@{->}[r]
    &0  
\
 }
\]
where $j$ is the identity map. We write $s \sim t$ if $f$ is an isomorphism. 

\s \label{partial-prop} we have the following properties of $ > $ on $\Sc(M)$
\begin{enumerate}
\item
If $s > t $ and $t > l$  then $s > l$; (obvious).
\item
If $s > t $ and $t > s $  then $s \sim t$; see \cite[2.4]{Y}.
\item
If $s,t \in \Sc(M)$ then there exists $u \in \Sc(M)$ such that $s > u$ and $t > u$; see \cite[2.6]{Y}.
\end{enumerate}

By \ref{partial-prop} it follows that if there is a minimal element in $\Sc(M)$ then it is a minimal element in 
$\Sc(M)$ (upto isomorphism). 
\begin{definition}
An AR-sequence ending at $M$ is the unique minimal element of $\Sc(M)$ (if it exists).
\end{definition}
For a more concrete description of AR-sequences see \cite[2.9]{Y}. 

\s The following two results are basic. The first is \cite[3.4]{Y}.

\begin{theorem} Let $(A, \m)$ be a Henselian \CM \ local ring. Let $M \in \CMa(A)$ be  non-free and  indecomposable.  
The following two conditions are equivalent:
\begin{enumerate}[ \rm(i)]
\item
$M$ is locally free on the punctured spectrum of $A$.
\item
There exists an AR-sequence ending at $M$.
\end{enumerate}
\end{theorem}

The second result is  \cite{A} (when $A$ is complete), \cite[4.22]{Y} (when $A$ has a canonical module) and \cite[Corollary 2]{HL} (in general).
\begin{theorem} Let $(A, \m)$ be a Henselian \CM \ local ring. If $A$ is of finite representation type then $A$ is an isolated singularity.
\end{theorem}

\section{A construction}
In this section we describe a construction that is essential to us. This was constructed in \cite{Pu}.
\s \label{const} Let $(A,\m)$ be a Henselian local ring with perfect residue field $k$. Let $\ov{k}$ be the algebraic closure of $k$. Let  
$$\C_k = \{E \mid E \ \text{is a finite extension of } \ k, \ \text{and} \ E \subseteq \ov{k} \}.$$
Order $\C_k$ with the  inclusion as partial order. Note that $\C_k$ is a directed set, for if $E, F \in \C_k$ then the composite field $EF \in \C_k$ and clearly $EF \supseteq E$ and $EF \supseteq F$. We prove
\begin{theorem}
\label{basic-c}[ \cite[4.2]{Pu}](with hypotheses as in \ref{const})  There exists a direct system of local rings $\{ (A^E,\m^E) \mid E \in \C_k\}$ such that
\begin{enumerate}[\rm (1)]
\item
$A^E$ is a finite flat extension with $\m A^E = \m^E$. Furthermore $A^E/\m^E  \cong E$ over $k$.
\item
$A^E$ is Henselian.
\item
For any $F, E \in \C_k$ with $F \subseteq E$ the maps in the direct system $\theta^E_F \colon A^F \rt A^E$ is flat and local with $\m^F A^E = \m^E$.
\end{enumerate}
\end{theorem}
The ring $T = \lim_{E\in \C_k} A^E$ will have nice properties.  
\s \label{const-1}\textbf{Construction-C.1:}  For every  $E \in \C_k$  we construct a ring $A^E$ as follows.  As $k$ is perfect, $E$ is a separable extension of $k$. So by primitive element theorem  $E = k(\alpha_E)$ for some $\alpha_E \in E$. Let 
\[
p_E(X) = p_{E,\alpha_E}(X) = \mbox{Irr}(\alpha_E, k),
\]
be the unique monic minimal polynomial of $\alpha_E$ over $k$. Let $f_E(X) = f_{E, \alpha_E}(X)$ be a monic polynomial in $A[X]$ such that $\ov{f_E(X)} = p_E(X)$.
Set
\[
A^E = \frac{A[X]}{(f_E(X))}.
\]
Our construction of course depends on choice of $\alpha_E$ and the choice of $f_E(X)$. We will simply fix one choice of $\alpha_E$ and $f_E(X)$.

\begin{remark}
If $A$ contains a field isomorphic  to $k$ then we can choose $A^E = A\otimes_k E$.  However note that in general, even if $A$ contains a field, it  need not contain a field isomorphic to $k$. 
\end{remark}

\s \label{const-2}\textbf{Construction-C.2:}
Let $k \subseteq F \subseteq E$ be a tower of fields.  In \cite[4.5]{Pu} we constructed a ring homomorphism $\theta^E_F \colon A^F \rt A^E$  such that the following holds:
\begin{proposition}\label{E-basic-2}\cite[4.6]{Pu}
(with hypotheses as in \ref{const-2})
\begin{enumerate}[\rm (i)]
\item
$\theta^E_F$ is a homomorphism of $A$-algebra's.
\item
$\theta^E_F$ is a local map
and
$\m^FA^E = \m^E$.
\item
$A^E$ is a flat $A^F$-module (via $\theta^E_F$).
\item
If $k \subseteq F \subseteq  E \subseteq L$ is a tower of fields then we have a commutative diagram 
\[
\xymatrix{ 
A^F
\ar@{->}[d]_{\theta^E_F} 
\ar@{->}[dr]^{\theta^L_F}
 \\ 
A^E
\ar@{->}[r]_{\theta^L_E} 
& A^L
}
\]
\end{enumerate}
\end{proposition}

\s \label{const-3} \textbf{Construction-C.3:} By \ref{E-basic-2} we have a directed system of rings 
$\{ A^E \}_{E \in \C_k}$.
Set 
\[
T = \lim_{E \in \C_k} A^E,
\]
and let $\theta_E \colon E \rt T$ be the maps such that for any $F\subseteq E$ in $\C_k$ we have $\theta_E \circ \theta^E_F = \theta_F$.
For $F \in \C_k$ set
$$ \C_F = \{E \mid E \ \text{is a finite extension of } \ F \}.$$
Then clearly $\C_F$ is cofinal in $\C_k$. Thus we have
\[
T = \lim_{E \in \C_F} A^E.
\]
We have the following properties of $T$.
\begin{theorem}\label{prop-T}[See \cite[4.8]{Pu}]
(with hypotheses as in \ref{const-3})
\begin{enumerate}[\rm (i)]
\item
$T$ is a Noetherian ring.
\item
$T$ is a flat $A$-module. 
\item
$T$ is a flat $A^F$-module for any $F \in \C_k$.
\item
The map $\theta_E$ is injective for any $E \in \C_k$. 
\item
By (iv)
we may write $T = \bigcup_{E \in \C_k}A^E$. Set $\m^T = \bigcup_{E \in \C_k}\m^E$.
Then $\m^T$ is the unique maximal ideal of $T$. 
\item
 $\m T = \m^T$.
\item
$T/\m^T \cong \ov{k}$.
\item
$T$ is a Henselian ring.
\end{enumerate}
\end{theorem}
The following result is definitely known to experts. We give a proof for the convenience of the reader.
\begin{lemma}
 \label{excellent}
 If $A$ is excellent then 
 \begin{enumerate}[\rm (1)]
 \item
 $A^E$ is excellent  for all $E \in \C_k$.
 \item
  $T = \lim_{E \in \C_k} A^E$ is excellent.
 \end{enumerate}
 \end{lemma}
\begin{proof}
(1) We have $A^E = A[X]/(f_E(X))$. So $A^E$ is excellent.

(2) In the directed system $\{ A^E \}_{E \in \C_k}$ each map $A^F \rt A^E$ (when $F\subseteq E$) is etale. 
 So by a result of \cite[5.3]{G} it follows that $T$ is excellent.
\end{proof}

The significance of $T$ is that certain crucial properties descend to a finite extension $E$ of $k$,  see
\cite[4.9]{Pu}.
\begin{lemma}\label{l-p-t}
(with hypotheses as above) 
\begin{enumerate}[\rm (1)]
\item
Let $M$ be a  $T$-module. Then there exists $E \in \C_k$ and an  $A^E$-module $N$ such that $M = N\otimes_{A^E} T$.
\item
Let $N_1,N_2$ be $A^E$-modules for some $E \in \C_k$. Suppose there is a $T$-linear map $f \colon N_1\otimes_{A^E} T \rt N_2 \otimes_{A^E} T$. Then there exists $K \in \C_k$ with $K \supseteq  E$ and an $A^K$-linear map $g \colon N_1\otimes_{A^E} A^K \rt N_2 \otimes_{A^E} A^K$ such that $f = g\otimes T$. Furthermore if $f$ is an isomorphism then so is $g$.
\end{enumerate}
\end{lemma}

We now relate finite representation property of our construction.
\begin{lemma}
 \label{et-fin-rep}
 Assume  $A$ is \CM, 
 excellent and of finite representation type. Then
 \begin{enumerate}[\rm (1)]
 \item
 $A^E$  is \CM \ of finite representation type for each $E \in \C_k$.
 \item
 $T = \lim_{E \in \C_k} A^E$ is \CM \ of finite representation type.
 \item
$\widehat{T}$,  the $\m^T$ completion of $T$, is \CM \ of finite representation type.
\item
If $A$ is Gorenstein then $A^E$  is Gorenstein for each $E \in \C_k$. Furthermore $T$, $\widehat{T}$ are  Gorenstein.
 \end{enumerate}
\end{lemma}
\begin{proof}
We first note that as $\m A^E = \m^E$. So $A^E$ is \CM, see \cite[Corollary, p.\ 181]{Mat}. Furthermore if $A$ is Gorenstein then so is $A^E$, see \cite[23.4]{Mat}.  Similarly as $\m T = \m^T$ we get 
$T$ is \CM \ (and is Gorenstein if $A$ is). So $\widehat{T}$ is also \CM \ (and is Gorenstein if $T$ is).

For (1), (2) see \cite[10.7]{LW}. For (3) use \ref{excellent} and \cite[10.10]{LW}.
\end{proof}

The following results on comparing AR-sequences is crucial for us.
\begin{lemma}\label{AR-compare}
 Let the setup be as in Lemma \ref{et-fin-rep}. Let $M^T$  be an indecomposable MCM $T$-module and let
$s^T \colon 0 \rt N^T \rt L^T \rt M^T \rt 0$ be an AR-sequence ending at $M^T$. By \ref{l-p-t} there exists $E \in \C_k$ and MCM Modules $A^E$-modules $M^E, N^E, L^E$ such that
\begin{enumerate}[\rm (1)]
\item
$M^E \otimes_{A^E} T = M^T$, $N^E \otimes_{A^E} T = N^T$ and $L^E \otimes_{A^E} T = L^T$.
\item
A short exact sequence, $s^E \colon 0 \rt N^E \rt L^E \rt M^E \rt 0$, of $A^E$-modules  such that $s^E \otimes_{A^E} T = s^T$.
\end{enumerate}
Then 
\begin{enumerate}[\rm (a)]
\item 
$s^E$ is the AR-sequence in $A^E$ ending at $M^E$.
\item
If $E \subseteq F$ then $s^F = s^E \otimes_{A^E} A^F$ is the AR-sequence in $A^F$ ending at 
$M^F = M^E \otimes_{A^E} A^F$.
\end{enumerate}
\end{lemma}
\begin{proof}
(a) As $N^T, M^T$ are indecomposable we get $N^E, M^E$ are indecomposable. Let $\beta$ be an AR-sequence in $A^E$ 
ending at $M^E$. Then $s^E > \beta$. So $s^T = s^E \otimes T > \beta \otimes T $. But $s^T$ is the AR-sequence ending at $M^T$. So $\beta\otimes T > s^T$. Therefore $s^E \otimes T \sim \beta \otimes T$ (see \ref{partial-prop}(2)). As $T$ is a faithfully flat $A^E$-algebra we get that $s^E \sim \beta$. The result follows.

(b) Note
\[
s^F \otimes_{A^F} T = (s^E \otimes_{A^E} A^F)\otimes_{A^F} T \cong s^T.
\]
The result follows from (a). 
\end{proof}

\section{Proof of our main result \ref{main-intro}}
In this section we prove our main result. We require several preparatory results to prove it.
Throughout this section $(A,\m)$ is an excellent \CM \ local ring of finite representation type with $k = A/\m$ perfect. Fix an algebraic closure $\ov{k}$ of $k$. Let  
$$\C_k = \{E \mid E \ \text{is a finite extension of } \ k, \ \text{and} \ E \subseteq \ov{k} \}.$$
For $E \in \C_k$ let $A^E$ be as in \ref{const-2}. If $k \subseteq F \subseteq E$ let $\theta^E_F \colon  A^F \rt A^E$ be as in \ref{const-2}. As discussed above $\{ A^E \mid  E \in \C_k \}$ forms a direct system of rings.
As before set $T = \lim_{E \in \C_k} A^E$. By \ref{et-fin-rep} we get that $A^E$ has finite representation type for each $E \in \C_k$. Furthermore $T$ and $\widehat{T}$ also have finite representation type.

\s \label{const-K.1}\textbf{Construction-K.1:} Let $k \subseteq F \subseteq E$.  As $A^E$ is a flat $A^F$-algebra we have an obvious map
$\eta^E_F \colon G(A^F) \rt G(A^E)$  given by $M \rt M\otimes_{A^F} A^E $. After tensoring with $\Q$ denote this map by 
$(\eta^E_F)_\Q$.
It is clear that we have a direct system of abelian groups $\{G(A^E) \}_{E \in \C_k}$. So we have 
an abelian group $\lim_{E \in \C_k}G(A^E) $ and natural maps $\eta_E \colon G(A^E) \rt \lim_{E \in \C_k}G(A^E)$.

Next we show
\begin{lemma}\label{inclusion}
Let $k \subseteq F \subseteq E$. Then the map $(\eta^E_F)_\Q \colon G(A^F)_\Q \rt G(A^E)_\Q $ is an inclusion of $\Q$-vector spaces.
\end{lemma}
\begin{proof}
We note that via $\theta^E_F \colon  A^F \rt A^E$ we get that $A^E$ is a finite free $A^F$-module, say of rank $r$. It follows that any MCM $A^E$-module is also an MCM $A^F$-module. So  we have the obvious map
$\phi \colon G(A^E) \rt G(A^F)$. 

Set $\delta = ( \phi\otimes \Q)\circ  (\eta^E_F)_\Q$. Let $M$ be a MCM $A^F$-module.   Then note that 
$\delta([M]) = r [M]$. So $\delta$ is an isomorphism. In particular $(\eta^E_F )_\Q $ is an inclusion. 
\end{proof}
As an immediate consequence we get Proposition \ref{inj-intro}, which we restate for the convenience of the reader. 
\begin{corollary}\label{eta}
Let $F \in \C_k$ be any. 
The map $(\eta_F)_\Q \colon G(A^F)_\Q \rt \lim_{E \in \C_k}G(A^E)_\Q$ is injective. 
\end{corollary}
\begin{proof}
See Chapter III, Exercise 19 in \cite{L}.
\end{proof}
\s \label{const-K.2}\textbf{Construction-K.2:} Let $E \in \C_k$. As $T$ is a flat $A^F$-algebra we have an obvious map
$\xi_E \colon G(A^E) \rt G(T)$  given by $M \rt M\otimes_{A^E} T $.  The maps $\xi_E$ are compatiable with $\eta^E_F$ whenever $k \subseteq F \subseteq E$.
So we have a natural map
\[
\xi \colon \lim_{E \in \C_k}G(A^E) \rt G(T).
\]
We restate Theorem \ref{main-intro} for the convenience of the reader.
\begin{theorem}
\label{main-body} $\xi$ is an isomorphism.
\end{theorem}

The proof of Theorem \ref{main-body} requires a few preliminaries.

\s \label{const-K.3}\textbf{Construction-K.3:} We know that $T$ is of finite representation type.
Let $\Ic_T = \{M_1,\ldots, M_s \}$.
  By \ref{l-p-t} we can choose $F \in \C_k$ and indecomposable MCM $A^F$-modules $ M_1^F,\ldots, M_s^F$ with
$M_i = M_i^F\otimes_{A^F} T$ for $i = 1,\ldots, s$.
Set 
$$\Ic_{A^F}^\prime = \{ M_1^F,\ldots, M_s^F \}.$$
Note $\Ic_{A^F}^\prime $ can be a proper subset of $\Ic_{A^F}$ the set of all indecomposable MCM $A^F$-modules. 
By \ref{AR-compare} we may further assume (after possibly taking a finite extension of $F$)
that there exists a finite subset $\Rc_{A_F}$ of  $\AR(A^F)$ such that $\Rc_{A_F}\otimes_{A^F}T = \AR(T)$.
Further note that  $\Rc_{A_F}$ need not equal $\AR(A^F)$.

Consider the set 
$$\C_F = \{E \mid E \ \text{is a finite extension of } \ F, \ \text{and} \ E \subseteq \ov{k} \}.$$
Then $\C_F$ is co-final in $\C_k$. Also for $E\in \C_F$ we may choose $\Ic_{A^E}^\prime $ with 
$\Ic_{A^E}^\prime = \Ic_{A^F}^\prime \otimes_{A^F} A^E $. Also set $\Rc_{A_E} = \Rc_{A_F} \otimes_{A^F} A^E$.

\begin{remark}
\label{surj} It is obvious that the natural map $G(A^E) \rt G(T)$  is surjective for each $E \in \C_F$. So
$\xi_E$ is surjective. So the map 
\[
\xi^\prime  \colon \lim_{E \in \C_F}G(A^E) \rt G(T), \quad \text{is surjective.}
\]
As $\C_F$ is co-final in $\C_k$ we get
\[
\lim_{E \in \C_F}G(A^E)  = \lim_{E \in \C_k}G(A^E)  \ \quad \text{and} \ \xi^\prime = \xi.
\]
So $\xi$ is surjective.
\end{remark}
To prove  $\xi$ is injective requires some more work.
\s \label{const-K.4}\textbf{Construction-K.4:} Let $E \in \C_F$.  Let $H(E)$ be the free abelian group generated on $\add(\Ic_{A^E}^\prime)$.  Let $H_0(E)$ be the subgroup of $H(E)$ generated by 
$$\{ X_1-X_2 + X_3 \mid  \text{ there is a sequence} \  0 \rt X_1 \rt X_2 \rt X_3 \rt 0 \ \text{in} \ \Rc_{A_E}  \}.$$
Set $D(E) = H(E)/H_0(E)$. By our construction we have
\begin{enumerate}
\item
If $F \subseteq E_1 \subseteq  E_2$ then the map $-\otimes_{A^{E_1}} A^{E_2}$ induces an isomorphism 
$D(E_1) \rt D(E_2)$.
\item
Let $E \in \C_F$. The map $-\otimes_{A^E} T$ induces an isomorphism 
$D(E) \rt G(T)$.
\item
We also have an obvious map $D(E) \rt G(A^E)$ for all $E \in \C_F$. 
\end{enumerate}
 We have a commutative diagram
\[
\xymatrix{ 
D(E)
\ar@{->}[d]_{\alpha_E} 
\ar@{->}[dr]^{\beta_E}
 \\ 
G(A^E)
\ar@{->}[r]_{\xi_E} 
& G(T)
}
\]
We have directed system $\{ D(E)\}_{E \in \C_F}$ where the maps are induced as in \ref{const-K.4}(1) (which are isomorphisms). The maps $\{ \alpha_E \}_{E \in \C_F}$ is a map of directed systems. Also $\{\beta_E \}_{E \in \C_F}$ and 
$\{\xi_E \}_{E \in \C_F}$ are compatible with the obvious maps.  So we have a commutative diagram
\[
\xymatrix{ 
\lim_{E \in \C_F} D(E)
\ar@{->}[d]_{\alpha} 
\ar@{->}[dr]^{\beta}
 \\ 
\lim_{E \in \C_F} G(A^E)
\ar@{->}[r]_{\xi} 
& G(T)
}
\]
\begin{remark}\label{surj-roy}
Note $\beta $ is an isomorphism. So $\alpha$ is injective. Also $\xi$ is surjective. To prove $\xi$ is an isomorphism it suffices to show $\alpha$ is surjective (and so $\alpha$ is an isomorphism).
\end{remark}

To prove $\alpha$ is surjective we need the following:
\begin{lemma}\label{lem-surj}
Let $R$ be a commutative ring.
Let $\Lambda$ be a directed set.
Let $\{W_i \}_{i \in \Lambda}$ and $\{V_i \}_{i \in \Lambda}$ be two directed system of $R$-modules with
maps $\alpha_i^j \colon W_i \rt W_j$ and $\beta_i^j \colon V_i \rt V_j$ for all $i <j$ in $\Lambda$.
We do not assume $W_i$ or $V_i$ are finitely generated $R$-modules.
Suppose we have a 
 map of direct systems  $\{\eta_i \colon W_i \rt V_i \}_{i \in \Lambda}$ . 
  Further assume that for each $i \in \Lambda$ and each $v \in V_i$ there exists $j > i$ and $w \in W_j$ such that $\eta_j(w) = \beta_i^j(v)$ (here $j$ might possibly depend on $v$). Then the map 
  $$ \eta \colon \lim_{i \in \Lambda}W_i \rt \lim_{i \in \Lambda}V_i \quad \text{is surjective.} $$
\end{lemma}
\begin{proof}
Let us recall the construction of $\lim_{i \in \Lambda}V_i$, see \cite[Theorem(III, 10.1)]{L}.
Let $F = \bigoplus_{i \in \Lambda} V_i$. Let $N$ be the submodule of $F$ generated by $x - \beta_i^j(x)$ for all $x \in V_i$, for all $j > i$  (and for all $i$). Then  $\lim_{i \in \Lambda}V_i = F/N$. Similarly we can construct $\lim_{i \in \Lambda}W_i $.

Let $\theta = \sum_{l = 1}^{r}x_{i_l} \in \lim_{i \in \Lambda}V_i $ with $x_{i_l} \in V_{i_l}$. By our hypothesis for each $i_l$ there exists $j_l > i_l$ and $y_{j_l} \in W_{j_l}$ with 
$$\eta_{j_l}( y_{j_l}) = \beta_{i_l}^{j_l}(x_{i_l}).$$
We note that 
$$ \eta(\sum_{l =1}^{r} y_{j_l}) = \sum_{l =1}^{r}\beta_{i_l}^{j_l}(x_{i_l}) = \theta.$$
So $\eta$ is onto.
\end{proof}

We now give 
\begin{proof}[Proof of \ref{main-body}]
By \ref{surj} we get $\xi$ is surjective. To show $\xi$ is an isomorphism, by remark \ref{surj-roy} it suffices to show that $\alpha$ is surjective. 

We have two direct systems $\{ D(E) \}_{E \in \C_F}$ and  $\{ G(A^E) \}_{E \in \C_F}$ and maps
of direct systems $\{\alpha_E \colon  D(E) \rt  G(A^E) \}_{E \in \C_F}$.
Let $M$ be a MCM $A^E$-module. Then $M\otimes_{A^E} T = \bigoplus_{i=1}^{s}M_i^{a_i}$ for some $a_i \geq 0$.
Set $V = \bigoplus_{i=1}^{s}(M_i^E)^{a_i}$ Then $V$ is an MCM $A^E$-module and $V\otimes_{A^E} T \cong M \otimes_{A^E} T$
as $T$-modules.
Then by
Lemma \ref{l-p-t} there exists $K \supseteq E$ such that $M\otimes_{A^E}A^K \cong V\otimes_{A^E}A^K $ as $A^K$-modules. We note that $ V\otimes_{A^E}A^K \in \add(\Ic^\prime_{A^K})$ and 
$$\alpha_K (V\otimes_{A^E}A^K ) = M\otimes_{A^E}A^K = \eta_E^K(M).$$
So our direct systems satisfy the hypotheses of Lemma \ref{lem-surj}. Thus $\alpha$ is surjective. 
\end{proof}
We now give
\begin{proof}[Proof of \ref{intro-cor}]
By \ref{inclusion} and \ref{main-body}   we have an injection $G(A)_\Q \rt G(T)_\Q$. Also $G(T) \cong G(\widehat{T})$. By \cite[10.17]{LW} $\widehat{T}$ 
is an ADE-singularity. The Grothendieck groups of ADE-singularities have been computed, see \cite[13.10]{Y}.

(2) We have
$$\dim_\Q G(A)_\Q \leq \dim_\Q G(\widehat{T})_\Q  \leq 3.$$ 
The result follows.

(1)  We have
$$\dim_\Q G(A)_\Q \leq \dim_\Q G(\widehat{T})_\Q  = 1.$$
 Also as $A$ is an isolated singularity of dimension $\geq 2$ we get that $A$ is a domain. So
 we have an obvious surjective map $\rank \colon G(Q) \rt \Z$ which maps $M$ to $\rank(M)$. It follows that
 $\dim_\Q G(A)_\Q  \geq 1$. The result follows.
\end{proof}
\section{An example}
We now give an example which proves two things:

\begin{enumerate}[ \rm (1)]
\item
If $\dim A$ is odd then  $G(A)_\Q$ can be proper subspace of $G(\widehat{T})_\Q$.
\item
In general we cannot compare torsion subgroups of $G(A)$ and $G(\widehat{T})$
\end{enumerate}

The example is $A = \mathbb{R}[[x,y]]/(x^2 + y^2)$. Note $\widehat{T} = \mathbb{C}[[x,y]]/(x^2 + y^2)$ is the 
$A_1$-singularity.
By  \cite[p.\ 134]{Y} we get that $G(A) = \Z \oplus \Z/2\Z$. While $G(\widehat{T}) = \Z^2$, see \cite[13.10]{Y}. This proves both our assertions.


\begin{thebibliography} {99}

\bibitem{A}
M.~Auslander,
\emph{Isolated singularities and existence of almost split sequences},
 In: Proc. ICRA IV, Springer Lecture Notes in Math, vol. 1178, pp. 194--241.h (1986)
 
 \bibitem{AR}
\bysame and I.~Reiten, 
\emph{ Grothendieck groups of algebras and orders},
 J. Pure Appl. Algebra 39 (1986), no. 1-2, 1--51





\bibitem{G}
S.~Greco,
\emph{Two theorems on excellent rings},
Nagoya Math. J. 60 (1976), 139--149.

\bibitem{HL}
C.~Huneke and G.~J.~Leuschke,
\emph{Two theorems about maximal Cohen-Macaulay modules},
Math. Ann. 324 (2002), no. 2, 391--404.

\bibitem{L}
S.~Lang,
\emph{Algebra},
 Revised third edition. Graduate Texts in Mathematics, 211. Springer-Verlag, New York, 2002.



\bibitem{LW}
G.~J.~Leuschke and R.~Wiegand,
\emph{Cohen-Macaulay representations},
Mathematical Surveys and Monographs, 181. American Mathematical Society, Providence, RI, 2012. 

 \bibitem{Mat}
H.~Matsumura, \emph{Commutative ring theory}, second ed., Cambridge
  Studies in Advanced Mathematics, vol.~8, Cambridge University Press,
  Cambridge, 1989.

\bibitem{Pu}
T.~J.~Puthenpurakal, 
\emph{Examples of non-commutative crepant resolutions of Cohen Macaulay normal domains},
J. Algebra 485 (2017), 77--96.

\bibitem{Y}
Y.~Yoshino, 
\emph{Cohen-Macaulay modules over Cohen-Macaulay rings},
 London Mathematical Society Lecture Note Series, 146. Cambridge University Press, Cambridge, 1990.
 
\end{thebibliography}
 \end{document}